\documentclass[12pt,reqno]{amsart}
 
\usepackage{amsmath,amssym b,mathrsfs}
\usepackage{graphicx,cite,times}

\usepackage{cite}
\usepackage{appendix}
\usepackage{enumerate}

\newtheorem{theo}{Theorem}[section]
\newtheorem{coro}[theo]{Corollary}
\newtheorem{lemm}[theo]{Lemma}
\newtheorem{prop}[theo]{Proposition}

\newtheorem{rema}[theo]{Remark}

\numberwithin{equation}{section}

\begin{document}

\title[stability estimates for an inverse problem for Schr\"odinger operators]{stability estimates for an inverse problem for Schr\"odinger operators at high frequencies from arbitrary partial boundary measurements}
\author{Xiaomeng Zhao and Ganghua Yuan}
\address{KLAS, School of Mathematics and Statistics, Northeast Normal University,
Changchun, Jilin, 130024, China}
\email{zhaoxm600@nenu.edu.cn}
\email{yuangh925@nenu.edu.cn}
\thanks{
The research of GY was supported in part by NSFC grants 11771074 and National Key R\&D Program of China (No. 2020YFA0714102).}


\begin{abstract}
In this paper, we study the partial data inverse boundary value problem for the Schr\"odinger operator at a high frequency $k\geq 1$ in a bounded domain with smooth boundary in $\mathbb R^n$, $n\geq3$. Assuming that the potential is known in a neighborhood of the boundary, we obtain the logarithmic stability when both Dirichlet data and Neumann data are  taken on arbitrary open subsets of the boundary where the two sets can be disjointed. Our results also show that the logarithmic stability can be improved to the one of H\"older type in the high frequency regime. To achieve those goals, we used a method by combining the CGO sulution, Runge approximation and Carleman estimate.

\end{abstract}

\maketitle

\section{introduction and statement of results}
Let $\Omega\subset\mathbb R^n, n\geq 3$, be a  bounded domain with $C^\infty$ boundary, let $k\geq 1$ and $q\in L^\infty(\Omega)$. We assume throughout this paper that $k^2$ is not an eigenvalue of $-\Delta+q$.  For any $f\in H^{\frac{1}{2}}(\partial\Omega)$, the Dirichlet problem
\begin{align}\label{mm}
	\begin{split}
	(-\Delta-k^2+q)u&=0 \quad\mbox{in} ~~\Omega,\\
	u&=f\quad\mbox{on}~~\partial\Omega,
	\end{split}
\end{align}
has a unique solution $u\in H^1(\Omega)$ and we can introduce the Dirichlet-to-Neumann (DtN) map
\begin{align*}
	\Lambda_q(k):H^{\frac{1}{2}}(\partial\Omega) &\rightarrow H^{-\frac{1}{2}}(\partial\Omega),\\
	f&\mapsto \partial_\nu u.	
\end{align*}
Here, $\nu$ stands for the outward unit normal vector to $\partial\Omega$ and $\partial_\nu u$ stands for $\nabla u\cdot\nu$.

Let $\Gamma_D$ and $\Gamma_N$ be arbitrary nonempty open subsets of $\partial\Omega$. We introduce a special  operator called the partial Dirichlet-to-Neumann (DtN in short) map defined by
\begin{align}
	\begin{split}
	\Lambda_q^{\Gamma_D,\Gamma_N}(k):\widetilde H^{\frac{1}{2}}(\Gamma_D) &\rightarrow H^{-\frac{1}{2}}(\Gamma_N),\\
	u|_{\partial\Omega}=f&\mapsto \partial_\nu u|_{\Gamma_N}.	
	\end{split}
\end{align}
where supp$f \subset \Gamma_D$
, $\widetilde H^{\frac{1}{2}}(\Gamma_D):=\{f\in H^{\frac{1}{2}}(\partial\Omega):{\rm supp}f\subset\Gamma_D\}$.

We are mainly interested in the study of the inverse problem of stably determining the potential $q$ in $\Omega$ from the knowledge of the partial  Dirichlet-to-Neumann map $\Lambda_q^{\Gamma_D,\Gamma_N}(k)$ associated to the problem (\ref{mm}). This kind of problem retrospects back to the original inverse conductivity problem, proposed by Calder\'on, arising in electrical impedance tomography where uniqueness for the linearization can be proven by using the complex exponential solutions \cite{calderon2006inverse}. Later on, the fundamental work by Sylvester and Uhlmann proved the global uniqueness for the inverse potential problem when $\Gamma_D=\Gamma_N=\partial\Omega$ by constructing almost complex exponential solutions, which also yields the global uniqueness of the inverse conductivity problem \cite{sylvester1987global}. Important progress on the partial data problem has been achieved in \cite{bukhgeim2002recovering} and \cite{ammari2004reconstruction}, \cite{imanuvilov2010calderon}, \cite{imanuvilov2011determination}, \cite{imanuvilov2011inverse}, \cite{imanuvilov2015neumann}, \cite{isakov2007uniqueness}, \cite{kenig2014calderon}, \cite{sjostrand2007calderon}, while there are many cases still remaining quite open, see \cite{kenig2014recent} for a review. Up to now, there is no result for both $\Gamma_D$ and $\Gamma_N$ taken on arbitrary open subsets of the boundary except for the case when $\Gamma_D=\Gamma_N$ (see e.g., \cite{ammari2004reconstruction} and \cite{garcia2021runge}), to the best of our knowledge.

Logarithmic stability for such kind of problem was obtained in  \cite{alessandrini1988stable} and was further verified to be optimal in \cite{mandache2001exponential}.  Meanwhile, inverse boundary value problems where the measurements are collected on a certain subset of the boundary are of particular interests and have their own importance in real applications. Various results have been developed in different settings of partial boundaries. Logarithmic stability estimates complementing the uniqueness result of \cite{ammari2004reconstruction}, were proved in \cite{fathallah2007stability}, see also  \cite{alessandrini2012single}, \cite{joud2009stability}, \cite{ruland2019quantitative},  while for the uniqueness result of \cite{isakov2007uniqueness}, logarithmic stability estimates were obtained in \cite{heck2016optimal}. For the uniqueness results of \cite{bukhgeim2002recovering} and \cite{sjostrand2007calderon}, the log-log type stability estimates were established in \cite{heck2006stability}, and in \cite{caro2014stability}, \cite{caro2016sef},  respectively.

The logarithmic stability estimates above indicate the difficulty in designing accurate reconstruction algorithms for the inverse Schr\"odinger potential problem, since small observation noise will lead to large error in reconstruction. Recently, it has been widely observed that in various inverse boundary value problems the stability estimate improves with growing wavenumber. This so-called increasing stability phenomenon for the inverse Schr\"odinger potential problem of (\ref{mm}) demonstrates that the stability estimate improves from a logarithmic type to a H\"older one if the wavenumber is sufficiently large. Such a result was verified in \cite{isakov2010increasings} by careful analysis in different frequency regions with full boundary measurements. Rigorous studies of increasing stability for different inverse potential problems can be found in \cite{garcia2021runge}, \cite{hrycak2004increased}, \cite{isakov2007increasedst}, \cite{isakovl2016increasing}, \cite{isakov2014increasing}, \cite{isakovw2013increasing}, \cite{krupchyk2019stability} and references therein.

As we know, the related increasing stability for the inverse boundary value problem of (\ref{mm}) in the full boundary data case are more  complete than the partial boundary data case. Thus,  there are still many open questions  related to partial data case, especially for arbitrary boundary data. Recently, some results are presented with arbitrary boundary data, such as \cite{krupchyk2019stability} and \cite{garcia2021runge}.  In \cite{krupchyk2019stability}, stability estimates were showed by the knowledge of the partial Robin-to-Dirichlet map at the fixed frequency $k>0$. In this paper, the Dirichlet boundary data are measured on arbitrarily small portion of the boundary $\partial \Omega$, while the Robin boundary data need to be measured on the whole boundary $\partial \Omega$. The complex geometric optics (CGO) solutions and Carleman estimate are used to prove the main results in  \cite{krupchyk2019stability}.  In \cite{garcia2021runge}, the approach of Runge approximation was used to prove stability improvement results in the high frequency limit for the acoustic Helmholtz equation by using partial DtN map with $\Gamma_D=\Gamma_N=\Gamma$, where $\Gamma$ is an arbitrary nonempty open subset of $\partial\Omega$.

Inspired by the references above, in this paper, we combine the methods of CGO solutions, Carleman estimates and Runge approximation to the inverse boundary problem (\ref{mm}) with a large wavenumber in three and higher dimensions. And we obtain increasing stability estimates to determine the potential with partial data. Both Dirichlet data and Neumann data are measured on arbitrarily small portions $\Gamma_D$, $\Gamma_N$ of the boundary, which can be disjoint sets.

To apply Runge approximation, we make additional assumptions:

(a1) $q\in L^\infty(\Omega)$ and zero is not a Dirichlet-eigenvalue of $\Delta-	q$ in $\Omega$,

(a2) ${\rm dist}(k^2,\Sigma_q)>ck^{2-n}$, for some $c\ll1$,
where $\Sigma_q$ denotes the set of the inverse of eigenvalues of the operator $T:=(\Delta-q)^{-1}$.

 As mentioned in \cite{garcia2021runge}, generically (a2) does not pose major restrictions, as it is always possible to find arbitrarily large values of $k$ such that the condition (a2) is fulfilled.
For more details for this, we can refer to \cite{garcia2021runge}.

Now, we state the main result of this paper as follows, which shows that the stability of the partial data inverse boundary problem increases as $k$ becomes large.

\begin{theo}\label{th}
Let $\Omega\subset\mathbb R^n$, $n\geq3$, be a bounded domain with $C^\infty$ boundary and let $\Gamma_D, \Gamma_N\subset\partial\Omega$ be arbitrary nonempty open subsets of $\partial\Omega$. Let $M>0$ and let $q_1, q_2\in L^\infty(\Omega)$ satisfy {\rm (a1)} and $\|q_j\|_{L^\infty(\Omega)}\leq M$, j=1,2. Assume that $q_1=q_2$ in $\mathcal O$, where $\mathcal O\subset\Omega$ is a neighborhood of $\partial\Omega$. Then, there exists a constant $C>1$ such that for all $k\geq 1$ satisfying {\rm (a2)} and $0<\delta:=\|\Lambda_{q_1}^{\Gamma_D,\Gamma_N}(k)-\Lambda_{q_2}^{\Gamma_D,\Gamma_N}(k)\|_{\widetilde H^{\frac{1}{2}}(\Gamma_D) \rightarrow H^{-\frac{1}{2}}(\Gamma_N)}<1$, we have
\begin{align}
	\|q_2-q_1\|_{H^{-1}(\Omega)}\leq e^{Ck}\delta+\frac{C}{\left(k+\log\frac{1}{\delta}\right)^{\frac{2}{n+2}}}.
\end{align}
Here $C>1$ depends on $\Omega$, $\mathcal O$, $M$, $n$ but independent of $k$.
\end{theo}
\begin{rema}
	The result of Theorem \ref{th} can be summarized informally by saying that at high frequencies, the stability estimate is of H\"older type, modulo an error term with a power-like decay as $k$ becomes large.
\end{rema}
\begin{rema}
The assumption $q_1=q_2$ in $\mathcal O$ is very realistic in most of the applications (see e.g., \cite{ammari2004reconstruction}).
\end{rema}
\begin{rema}
Theorem \ref{th} also implies stability for the  Schr\"odinger operators with $k=0$.  For this case, we use the DtN map to determine the potential $q$ for the equation $(-\triangle+q)u=0$ by assuming that $0$ is not a Dirichlet eigenvalue of $-\triangle+q$ in $\Omega$.
To deal with this problem, we fix a large enough $k_0$ and transform $(-\triangle+q)u=0$ into $(-\triangle-k_0^2+(k_0^2+q))u=0$, then by Theorem \ref{th} we can obtain the corresponding stability result for the  Schr\"odinger operators with $k=0$.
\end{rema}
\begin{rema}
To the best of our knowledge, Theorem \ref{th} is the first stability estimate for the Schr\"odinger equation with both Dirichlet data and Neumann data taken on arbitrary open subsets of the boundary.
\end{rema}
Assuming that the potentials $q_1$ and $q_2$ enjoy some additional regularity properties and a priori bounds, we can obtain the following  corollary of Theorem \ref{th}.
\begin{coro}\label{co}
Let $\Omega\subset\mathbb R^n$, $n\geq3$, be a bounded domain with $C^\infty$ boundary and let $\Gamma_D, \Gamma_N\subset\partial\Omega$ be arbitrary nonempty open subsets of $\partial\Omega$. Let $M>0$,$s>\frac{n}{2}$ and let $q_1, q_2\in H^s(\Omega;\mathbb R)$ satisfy {\rm (a1)} and $\|q_j\|_{H^s(\Omega)}\leq M$, j=1,2. Assume that $q_1=q_2$ in $\mathcal O$, where $\mathcal O\subset\Omega$ is a neighborhood of $\partial\Omega$. Then, there exists a constant $C>1$ such that for all $k\geq 1$ satisfying {\rm (a2)} and $0<\delta:=\|\Lambda_{q_1}^{\Gamma_D,\Gamma_N}(k)-\Lambda_{q_2}^{\Gamma_D,\Gamma_N}(k)\|_{\widetilde H^{\frac{1}{2}}(\Gamma_D) \rightarrow H^{-\frac{1}{2}}(\Gamma_N)}<1$, we have
	\begin{align}
		\|q_2-q_1\|_{L^\infty(\Omega)}\leq \left(e^{Ck}\delta+\frac{C}{\left(k+\log\frac{1}{\delta}\right)^{\frac{2}{n+2}}}\right)^{\frac{s-\frac{n}{2}}{2(s+1)}}.
	\end{align}
Here $C>1$ depends on $\Omega$, $\mathcal O$, $M$, $n$, $s$ but independent of $k$.
\end{coro}

The rest of the paper is organized as follows: after briefly recalling some auxiliary results in Section 2, we turn to the quantitative unique continuation estimate which will be proved by Carleman estimates for the operator $-\Delta-k^2$, $k\geq1$ in Section 3. In Section 4, we prove partial data stability for the Schr\"odinger equation at a large frequency with a priori information in a boundary layer. Finally, in Appendix, we give a sketch of the proof of Lemma \ref{c1}, the Carleman estimate which will be used in the paper.

\section{preliminaries}
To begin with, we state the existence of CGO solutions for (\ref{mm}). These special solutions are first constructed in \cite{sylvester1987global}. Another construction based on the Fourier series is given in  \cite{hahner1996periodic}.
\begin{lemm}\label{le1}
	Let $\widetilde\Omega\subset\mathbb R^n$ be a bounded domain. Let $q \in L^\infty(\widetilde\Omega)$ and $k \geq 0$. Then there are constants $C_0>0$ and $C_1>0$, depending on $\widetilde\Omega$ and $n$ only, such that for all $\zeta \in \mathbb{C}^n$, $\zeta\cdot\zeta=k^2$, and $|{\rm Im}\zeta|\geq\max\{C_0\|q\|_{L^\infty(\widetilde\Omega)}, 1\}$, the equation $$(-\Delta-k^2+q)u=0 \quad\mbox{in} ~~\widetilde\Omega$$has a solution $$u(x)=e^{{\rm i}\zeta\cdot x}(1+r(x)),$$where $r\in L^2(\widetilde\Omega)$ satisfies$$\|r\|_{L^2(\widetilde\Omega)}\leq\frac{C_1}{|\rm{Im}\zeta|}\|q\|_{L^\infty(\widetilde\Omega)}.$$
\end{lemm}

Let us recall the following standard elliptic regularity result from \cite[Theorem 7.1]{zworski2022semiclassical}.
\begin{lemm}\label{le2}
	Let $\Omega \subset\subset \widetilde{\Omega}\subset \mathbb{R}^n$, $q\in L^\infty(\widetilde{\Omega})$, and $k\geq0$. Let $u\in L^2(\widetilde{\Omega})$ be a solution to $$(-\Delta-k^2+q)u=0 \quad\mbox{in} ~~\widetilde{\Omega}.$$Then $u\in H^2(\Omega)$ and we have the following bounds,$$\|u\|_{H^1(\Omega)}\leq C(1+k)\|u\|_{L^2(\widetilde\Omega)}.$$
\end{lemm}
Next we introduce the Runge approximation in \cite{garcia2021runge} for the Schr\"odinger equation with a high frequency under the assumptions (a1) and (a2).

\begin{lemm}\label{le4}
Let $\Omega_1, \widetilde\Omega_1, \Omega_2\subset\mathbb{R}^n$ be  bounded Lipschitz domains such that $\Omega_1\Subset\widetilde\Omega_1\Subset\Omega_2$ and $\Omega_2\backslash\overline{\Omega_1}$ is connected. Let $\Gamma_D$ be a nonempty, open subset of $\partial\Omega_2$. Let $q$ satisfies ${\rm (a1)}$ in $\Omega_2$. Then there exist constants $\mu>1$ and $C>1$ depending on $n, \Omega_2, \Omega_1, \|q\|_{L^\infty(\Omega)}$ such that for any solution $v\in H^1(\widetilde\Omega_1)$ of $$(-\Delta-k^2+q)v=0 \quad\mbox{in} ~~\widetilde\Omega_1,$$with $k\geq1$ satisfying ${\rm(a2)}$, and any $\epsilon>0$, there exists a solution $u$ to$$(-\Delta-k^2+q)u=0 \quad\mbox{in} ~~\Omega_2$$with $u|_{\partial\Omega_2}\in \widetilde H^{\frac{1}{2}}(\Gamma_D)$ such that $$\|u-v\|_{L^2(\Omega_1)}\leq \epsilon \|v\|_{H^1(\widetilde\Omega_1)}, $$ $$\quad \|u\|_{H^{\frac{1}{2}}(\Gamma_D)}\leq Ce^{Ck}\epsilon^{-\mu}\|v\|_{L^2(\Omega_1)}.$$
\end{lemm}

The following result contains a global estimate for the homogeneous Dirichlet problem depending on dist($k^2,\Sigma_q$), see \cite[Lemma 1.3]{garcia2021runge}.

\begin{lemm}\label{le3}
	Let $\Omega\subset\mathbb{R}^n$ be a bounded Lipschitz domain. Let $q$ satisfies ${\rm (a1
	)}$ in $\Omega_2$. Then there is a discrete set $\Sigma_q\subset\mathbb R$ such that for every $k\geq1$ satisfying ${\rm (a2)}$, there exists a unique solution $u\in H^1(\Omega)$ of
\begin{align*}
	(-\Delta-k^2+q)u&=f \quad\mbox{in} ~~\Omega,\\
	u&=0 \quad\mbox{on}~~\partial\Omega,
\end{align*}
where $f\in L^2(\Omega)$. In addition, there is a constant $C>0$ depending on $\Omega$ and $\|q\|_{L^\infty(\Omega)}$ such that$$\|u\|_{L^2(\Omega)}\leq Ck^n\|f\|_{L^2(\Omega)},$$ $$
  \quad\|u\|_{H^1(\Omega)}\leq Ck^{n+1}\|f\|_{L^2(\Omega)}.$$
\end{lemm}

\section{A unique continuation estimate for the Schr\"odinger equation at high frequencies}

In this section, we prove a unique continuation result for the solution of the problem (\ref{mm}) with homogeneous boundary condition which will present a crucial step in the proof of the main results.

Let $\Omega\subset\mathbb R^n$ be a bounded domain with $C^\infty$ boundary, and let $\nu$ be the outward unit normal vector to $\partial\Omega$. Let$$P(h,E)=-h^2\Delta-E,$$where $0<h\leq 1$ and $0\leq E\leq 1$. Letting $\Phi\in C^\infty(\overline\Omega;\mathbb{R})$, we set$$P_{\Phi}(h,E)=e^{\frac{\Phi}{h}}\circ P(h,E)\circ e^{-\frac{\Phi}{h}}.$$
Our starting point is the following  Carleman estimate which is due to \cite{krupchyk2019stability}, see also\cite{Fursikov1996Control},
\cite{le2012carleman}.

\begin{lemm}\label{c1}
	Let  $\Psi\in C^\infty(\overline\Omega;\mathbb R)$ be such that $\Psi\geq0$  on $\overline\Omega$ and $|\nabla\Psi|>0$ on $\overline\Omega$, and set $\Phi=e^{\gamma\Psi}$. Then there exist $\gamma_0> 0$, $h_0> 0$ and $C>0$, such that for all $\gamma\geq\gamma_0$, $0<h\leq h_0$, $0\leq E\leq 1$ and $v\in H^2(\Omega)$ satisfying $v=0$ on $\partial\Omega$, we have
	\begin{align}\label{app}
		\begin{split}
		h\left(\gamma^4\int_\Omega\Phi^3|v|^2\,dx\right.+&\left.\gamma^2\int_\Omega\Phi|h\nabla v|^2\,dx\right)\\
		\leq &C\left(\int_\Omega|P_\Phi(h,E)v|^2\,dx+h\gamma\int_{\partial\Omega}\Phi\partial_\nu\Psi|h\partial_\nu v|^2\,dS\right).
		\end{split}
	\end{align}
\end{lemm}
The proof of Lemma \ref{c1} is similar to \cite{krupchyk2019stability} except for  some boundary terms.
For the convenience of the readers, we give the proof in Appendix.

Next we will give a Carleman estimate for the operator $P(h,E)=-h^2\Delta-E$ by choosing a specific  $\Phi$ in Lemma \ref{c1}.

\begin{lemm}\label{c2}
	Let $\emptyset\neq \Gamma_N\subset\partial\Omega$ be open and let $\Psi\in C^\infty(\overline\Omega;\mathbb R)$ be such that $\Psi\geq0$  on $\overline\Omega$, $|\nabla\Psi|>0$ on $\overline\Omega$ and $\partial_\nu\Psi|_{\partial\Omega\backslash\Gamma_N}\leq0$. Let $\Phi=e^{\beta_0\Psi}$, $\beta_0\gg1.$ Then there exists $0<h_0\leq 1$ and $C>0$ such that for all  $0\leq E\leq 1$, $k\geq1$, $0<h\leq \frac{h_0}{k}$ and all $u\in H^2(\Omega)$ satisfying $u=0$ on $\partial\Omega$, we have
	\begin{align}\label{e7}
		h\int_\Omega e^{\frac{2\Phi}{h}}\left(|u|^2+|h\nabla u|^2\right)\,dx\leq C\left(\int_\Omega e^{\frac{2\Phi}{h}}|P(h,E)u|^2\,dx+h^3\int_{\Gamma_N}e^{\frac{2\Phi}{h}}|\partial_\nu u|^2\,dS\right).
	\end{align}
\end{lemm}
\begin{proof}
	Putting $v=e^{\frac{\Phi}{h}}u$, we have $P_\Phi(h,E)v=e^{\frac{\Phi}{h}}P(h,E)u$ and $v|_{\partial\Omega}=0$. Then we can apply Lemma \ref{c1} and obtain
	\begin{align}\label{e1}
		h\left(\int_\Omega\Phi^3e^{\frac{2\Phi}{h}}|u|^2\,dx\right.+&\left.\int_\Omega\Phi e^{\frac{2\Phi}{h}}|h\nabla u+\beta_0\Phi  u\nabla\Psi|^2\,dx\right)\nonumber\\
		\leq &C\left(\int_\Omega e^{\frac{2\Phi}{h}}|P(h,E)u|^2\,dx+h\int_{\partial\Omega}e^{\frac{2\Phi}{h}}\Phi\partial_\nu\Psi|h\partial_\nu u|^2\,dS\right)\nonumber\\
		\leq &C\left(\int_\Omega e^{\frac{2\Phi}{h}}|P(h,E)u|^2\,dx+h^3\int_{\Gamma_N}e^{\frac{2\Phi}{h}}|\partial_\nu u|^2\,dS\right).
	\end{align}
	Here we have used $u|_{\partial\Omega}=0$ and $\partial_\nu\Psi|_{\partial\Omega\backslash\Gamma_N}\leq0$.
	
	By (\ref{e1}), we see that
	\begin{align*}
		h\int_\Omega e^{\frac{2\Phi}{h}}|h\nabla u|^2\,dx\leq &Ch\left(\int_\Omega e^{\frac{2\Phi}{h}}|h\nabla u+\beta_0\Phi  u\nabla\Psi|^2\,dx+\int_\Omega e^{\frac{2\Phi}{h}}|\beta_0\Phi  u\nabla\Psi|^2\,dx\right)\\
		\leq &Ch\left(\int_\Omega e^{\frac{2\Phi}{h}}|h\nabla u+\beta_0\Phi  u\nabla\Psi|^2\,dx+\int_\Omega e^{\frac{2\Phi}{h}}\Phi ^3|u|^2\,dx\right)\\
		\leq &C\left(\int_\Omega e^{\frac{2\Phi}{h}}|P(h,E)u|^2\,dx+h^3\int_{\Gamma_N}e^{\frac{2\Phi}{h}}|\partial_\nu u|^2\,dS\right).
	\end{align*}
	Thus, we can obtain
	$$h\int_\Omega e^{\frac{2\Phi}{h}}\left(|u|^2+|h\nabla u|^2\right)\,dx\leq C\left(\int_\Omega e^{\frac{2\Phi}{h}}|P(h,E)u|^2\,dx+h^3\int_{\Gamma_N}e^{\frac{2\Phi}{h}}|\partial_\nu u|^2\,dS\right),$$which completes the proof.
\end{proof}

To apply Lemma \ref{c2} in this paper, we need the following result on existence of a weight function $\Psi$ with special properties (see e.g., \cite{Fursikov1996Control}).

\begin{lemm}\label{c3}
	Let $\emptyset\neq\Gamma_N\subset\partial\Omega$ be an arbitrary open subset. Then there exists $\Psi\in C^\infty(\overline\Omega)$ such that
	\begin{align*}
		&\Psi(x)>0, \quad x\in \Omega, \quad\quad\quad \quad|\nabla\Psi(x)|>0,\quad x\in \overline\Omega,\\
		&\Psi(x)=0, \quad x\in\partial\Omega \backslash\Gamma_N, \quad\quad\partial_\nu\Psi\leq0, \quad x\in\partial\Omega\backslash\Gamma_N.
	\end{align*}
\end{lemm}
Before presenting the unique continuation result, we introduce some  notations. Let $\mathcal O\subset\Omega$ be an arbitrary neighborhood of $\partial\Omega$.
We consider three open subsets $\mathcal O_j (j=1,2,3)$ of $\mathcal O$, which are also neighborhood of $\partial\Omega$ with $C^{\infty}$ boundary,  such that
$$\partial\Omega\subset \partial\mathcal O_j,\quad  \overline{\mathcal O}_{j+1}\subset \mathcal O_j, \quad \overline{\mathcal O}_j\subset \mathcal O, \quad  j=1,2,3.$$
 In addition, we  set $\Gamma_\sharp=\partial\mathcal O\backslash\partial\Omega $.
 Consequently,
 $$\partial\mathcal O=\partial\Omega \cup \Gamma_\sharp, \quad \partial\Omega \cap \Gamma_\sharp=\emptyset.$$
And we may assume, without of generality, that $\Gamma_\sharp$ is $C^{\infty}$-smooth.

Let $q\in L^\infty(\Omega)$, $k\geq1$, and let $u\in H^2(\Omega)$ be such that
\begin{align}\label{e2}
	\begin{split}
		(-\Delta-k^2+q)u&=0 \quad\mbox{in} ~~\mathcal O,\\
	u&=0 \quad\mbox{on}~~\partial\Omega.
	\end{split}
\end{align}

We are able to prove the following local unique continuation estimate which is an analog of \cite[Proposition 3.1]{krupchyk2019stability}, see also  \cite{bellassoued2009logarithmic}, \cite{joud2009stability}.
\begin{prop}\label{pp}
	There are constants $0<h_0\leq 1$, $C>0$, $\alpha_1>0$ and  $\alpha_2>0$ such that for all $k\geq1$, $0<h\leq\frac{h_0}{k}$ and all $u\in H^2(\Omega)$ satisfying (\ref{e2}), we have
	\begin{align}\label{ee}
		\|u\|_{H^1(\mathcal O_2\backslash\mathcal O_3)}\leq C\left(e^{-\frac{\alpha_1}{h}}\|u\|_{H^1(\Omega)}+e^{\frac{\alpha_2}{h}}\|u\|_{H^2(\Omega)}^{\frac{1}{2}}\|\partial_\nu u\|_{H^{-\frac{1}{2}}(\Gamma_N)}^{\frac{1}{2}}\right).
	\end{align}
\end{prop}
\begin{proof}
	Letting $h>0$, we rewrite (\ref{e2})  as follows,
	\begin{align}\label{ee1}
		\begin{split}
			(-h^2\Delta-(hk)^2+h^2q)u&=0 \quad\mbox{in} ~~\mathcal O,\\
	u&=0 \quad\mbox{on}~~\partial\Omega.
		\end{split}
	\end{align}
	Thanks to Lemma \ref{c3} there exists $\psi\in C^\infty(\overline{\mathcal O})$ such that
	 \begin{align}
		&\psi(x)>0, \quad x\in \mathcal O, \quad\quad|\nabla\psi(x)|>0,\quad x\in \overline{\mathcal O},\label{e3}\\
		&\psi(x)=0, \quad x\in\Gamma_\sharp, \quad\quad\partial_\nu\psi\leq0, \quad x\in\partial\mathcal O\backslash\Gamma_N.\label{e4}
	\end{align}
Here we have noted that $\Gamma_\sharp=\partial\mathcal O\backslash\partial\Omega\subset\partial\mathcal O\backslash\Gamma_N$.
	Let
	\begin{align}\label{e8}
		\varphi=e^{\alpha\psi},
	\end{align}
	with $\alpha>0$ sufficiently large as in Lemma \ref{c2}.

Now we will borrow some ideas from \cite{krupchyk2019stability} and \cite{bellassoued2009logarithmic}, \cite{bellassoued2021stably},\cite{joud2009stability}. Recalling the first part of (\ref{e3}) and taking into account that $\mathcal O_1\subset \mathcal O$, we deduce that there exists a constant $m_0>0$ such that
\begin{align}\label{e5}
	\psi(x)\geq2m_0, \quad x \in \mathcal O_2\backslash\mathcal O_3.
\end{align}
Moreover, since $\psi(x)=0$ on $\Gamma_\sharp$, there exists a small neighborhood $W_\sharp$ of $\Gamma_\sharp$ such that
\begin{align}\label{e6}
	\psi(x)\leq m_0, \quad x \in W_\sharp, ~~W_\sharp\cap\overline{\mathcal O}_1=\emptyset.
\end{align}
Let $ W_\sharp^*\subset\ W_\sharp$ be an arbitrary fixed neighborhood of the subboundary part $\Gamma_\sharp$. In order to apply (\ref{e7}), it is necessary to introduce a cut-off function $\chi_1$ satisfying $0\leq \chi_1\leq1$, $\chi_1\in C^\infty(\mathcal O)$ and
\begin{align*}
	\chi_1(x)=
	\begin{cases}
		1,\quad x\in\mathcal O\backslash W_\sharp,\\
		0,\quad x\in W_\sharp^*.
	\end{cases}
\end{align*}
Setting $\widetilde u=\chi_1 u$ and noting that $u$ satisfies (\ref{ee1}), we obtain the following system
\begin{align}\label{e9}
		\begin{split}
			(-h^2\Delta-(hk)^2+h^2q)\widetilde u&=[-h^2\Delta,\chi_1 ]u
		\quad\mbox{in} ~~\mathcal O,\\
	\widetilde u&=0 \quad\mbox{on}~~\partial\mathcal O.
		\end{split}
	\end{align}
	Applying the Carleman estimate (\ref{e7}) for the operator $P(h,(hk)^2)=-h^2\Delta-(hk)^2$ on the domain $\mathcal O$, with the Carleman weight $\varphi$ given by (\ref{e8}), and $\widetilde u\in H^2(\mathcal O)$ satisfying (\ref{e9}), we obtain that there exist $0<h_0\leq1$ and $C>0$, such that for all $k\geq1$, $0\leq h\leq\frac{h_0}{k}$,
	\begin{align}\label{e010}
		h\int_\mathcal O e^{\frac{2\varphi}{h}}\left(|\widetilde u|^2+|h\nabla\widetilde u|^2\right)\,dx&\leq C\left(\int_\mathcal O e^{\frac{2\varphi}{h}}|P(h,(hk)^2)\widetilde u|^2\,dx+h^3\int_{\Gamma_N}e^{\frac{2\varphi}{h}}|\partial_\nu \widetilde u|^2\,dS\right)\nonumber\\
		&\leq C\left(\int_\mathcal O e^{\frac{2\varphi}{h}}\left|[-h^2\Delta,\chi_1 ]u-h^2q\widetilde u\right|^2\,dx+h^3\int_{\Gamma_N}e^{\frac{2\varphi}{h}}|\partial_\nu \widetilde u|^2\,dS\right).
	\end{align}
	Then, for $h$ sufficiently small, we directly get from (\ref{e010}) that
	\begin{align}\label{e10}
		h\int_\mathcal O e^{\frac{2\varphi}{h}}\left(|\widetilde u|^2+|h\nabla\widetilde u|^2\right)\,dx\leq C\left(\int_\mathcal O e^{\frac{2\varphi}{h}}\left|[-h^2\Delta,\chi_1 ]u\right|^2\,dx+h^3\int_{\Gamma_N}e^{\frac{2\varphi}{h}}|\partial_\nu \widetilde u|^2\,dS\right).
	\end{align}
	
As	$\chi_1(x)=1$ in $\mathcal O\backslash W_\sharp\supset\mathcal O_2\backslash\mathcal O_3$ and in view of (\ref{e5}), we get
	\begin{align}\label{e11}
		h\int_\mathcal O e^{\frac{2\varphi}{h}}\left(|\widetilde u|^2+|h\nabla\widetilde u|^2\right)\,dx&\geq he^{\frac{2}{h}e^{2m_0\alpha}}\int_{\mathcal O_2\backslash \mathcal O_3}\left(|u|^2+|h\nabla u|^2\right)\,dx\nonumber\\
		&:=he^{\frac{2}{h}e^{2m_0\alpha}}\|u\|^2_{H^1_{scl}(\mathcal O_2\backslash\mathcal O_3)},
	\end{align}	
	where $\|u\|^2_{H^1_{scl}(\Omega)}=\|u\|^2_{L^2(\Omega)}+\|h\nabla u\|^2_{L^2(\Omega)}$.
	
By the definition of $\chi_1$, we can see that $${\rm supp}([-h^2\Delta,\chi_1 ])\subset W_\sharp\backslash\ W_\sharp^*.$$
Consequently, noting also (\ref{e6}), we have
	\begin{align}\label{e12}
		\int_\mathcal O e^{\frac{2\varphi}{h}}\left|[-h^2\Delta,\chi_1 ]u\right|^2\,dx&\leq Ch^4e^{\frac{2}{h}e^{m_0\alpha}}\int_{W_\sharp\backslash\ W_\sharp^*}\left(|u|^2+|\nabla u|^2\right)\,dx\nonumber\\
		&\leq Ch^4e^{\frac{2}{h}e^{m_0\alpha}}\|u\|^2_{H^1(\Omega)}.
	\end{align}
Further, by applying the trace theorem, we obtain
\begin{align}\label{e13}
	h^3\int_{\Gamma_N}e^{\frac{2\varphi}{h}}|\partial_\nu \widetilde u|^2\,dS&\leq Ch^3e^{\frac{2}{h}e^{\alpha\|\psi\|_\infty}}\int_{\Gamma_N}|\partial_\nu  u|^2\,dS\nonumber\\
	&\leq Ch^3e^{\frac{2}{h}e^{\alpha\|\psi\|_\infty}}\|\partial_\nu u\|_{H^{-\frac{1}{2}}(\Gamma_N)}\|\partial_\nu u\|_{H^{\frac{1}{2}}(\Gamma_N)}\nonumber\\
	&\leq Ch^3e^{\frac{2}{h}e^{\alpha\|\psi\|_\infty}}\|u\|_{H^2(\Omega)}\|\partial_\nu u\|_{H^{-\frac{1}{2}}(\Gamma_N)}.
\end{align}

Putting	(\ref{e11}), (\ref{e12}) and (\ref{e13}) together, in view of (\ref{e10}), we have
\begin{align*}
	he^{\frac{2}{h}e^{2m_0\alpha}}\|u\|^2_{H^1_{scl}(\mathcal O_2\backslash\mathcal O_3)}\leq C\left(h^4e^{\frac{2}{h}e^{m_0\alpha}}\|u\|^2_{H^1(\Omega)}+h^3e^{\frac{2}{h}e^{\alpha\|\psi\|_\infty}}\|u\|_{H^2(\Omega)}\|\partial_\nu u\|_{H^{-\frac{1}{2}}(\Gamma_N)}
\right).
\end{align*}
Let $$\alpha_1=	e^{2m_0\alpha}-e^{m_0\alpha}>0$$
and$$	\alpha_2=	e^{\alpha\|\psi\|_\infty}-e^{2m_0\alpha}>0.$$
We conclude that 	
\begin{align*}
	\|u\|_{H^1_{scl}(\mathcal O_2\backslash\mathcal O_3)}\leq C\left(h^{\frac{3}{2}}e^{-\frac{\alpha_1}{h}}\|u\|_{H^1(\Omega)}+he^{\frac{\alpha_2}{h}}\|u\|_{H^2(\Omega)}^{\frac{1}{2}}\|\partial_\nu u\|_{H^{\frac{1}{2}}(\Gamma_N)}^{-\frac{1}{2}}
\right),
\end{align*}	
and therefore,	
\begin{align*}
	\|u\|_{H^1(\mathcal O_2\backslash\mathcal O_3)}&\leq\frac{1}{h}\|u\|_{H^1_{scl}(\mathcal O_2\backslash\mathcal O_3)}\\
	&\leq C\left(e^{-\frac{\alpha_1}{h}}\|u\|_{H^1(\Omega)}+e^{\frac{\alpha_2}{h}}\|u\|_{H^2(\Omega)}^{\frac{1}{2}}\|\partial_\nu u\|_{H^{\frac{1}{2}}(\Gamma_N)}^{-\frac{1}{2}}
\right),
\end{align*}	
for all	$k\geq1$, $0\leq h\leq\frac{h_0}{k}$, and some $\alpha_1, \alpha_2>0$ independent of $h$ and $k$.  This completes the proof of  Proposition \ref{pp}.

\end{proof}

\section{stability estimate for the inverse problem}
The main purpose of this section is to prove Theorem \ref{th} and Corollary \ref{co}. First we recall the definition of the DtN map $\Lambda_q^{\Gamma_D,\Gamma_N}(k)$:
\begin{align}\label{a1}
	\begin{split}
	\Lambda_q^{\Gamma_D,\Gamma_N}(k):\widetilde H^{\frac{1}{2}}(\Gamma_D) &\rightarrow H^{-\frac{1}{2}}(\Gamma_N),\\
	u|_{\partial\Omega}=f&\mapsto \partial_\nu u|_{\Gamma_N},	
	\end{split}
\end{align}
where ${\rm supp}f\subset \Gamma_D$
, $\widetilde H^{\frac{1}{2}}(\Gamma_D):=\{f\in H^{\frac{1}{2}}(\partial\Omega):{\rm supp}f\subset\Gamma_D\}$, $\Gamma_D$ and $\Gamma_N$ are arbitrary nonempty open subsets of $\partial\Omega$ and $u$ is the unique solution of the equation (\ref{mm}).

Let $\xi\in\mathbb{R}^n$ and $\omega_1$, $\omega_2\in\mathbb{R}^n$ be such that $$\xi\cdot\omega_1=\xi\cdot\omega_2=\omega_1\cdot\omega_2=0 \quad \mbox{and} \quad| \omega_1|=|\omega_2|=1.$$
We set $$\zeta_1=-\frac{\xi}{2}+\sqrt{k^2+a^2-\frac{|\xi|^2}{4}}\omega_1+{\rm i}a\omega_2,$$
$$\zeta_2=-\frac{\xi}{2}-\sqrt{k^2+a^2-\frac{|\xi|^2}{4}}\omega_1-{\rm i}a\omega_2,$$
and require that $k^2+a^2\geq\frac{|\xi|^2}{4}$, see \cite{isakov2014increasing}.
 It is clear that$$\zeta_j\cdot\zeta_j=k^2,\quad\zeta_1+\zeta_2=-\xi,\quad|\zeta_j|^2=k^2+2a^2, \quad\mbox{for } j=1,2.$$

Let $\widetilde\Omega$ be a bounded domain such that $\Omega\subset\subset\widetilde\Omega\subset\subset\mathbb{R}^n$, and let us extend $q_1$ and $q_2$ by zero to $\mathbb{R}^n\backslash\Omega$ and denote the extensions by $q_1$ and $q_2$ again. Then, Lemma \ref{le1} guarantees that when $a=|{\rm Im}\zeta_j|$ satisfies
\begin{align*}
	a\geq\max\{C_0M,1\}, \quad j=1,2,
\end{align*}
there exist the complex geometric optics solutions $u_j$ to
\begin{align}\label{m1}
(-\Delta-k^2+q_j)u_j=0 \quad \mbox{in} ~~\widetilde\Omega,	
\end{align}
 in the form
 \begin{align}\label{m2}
 u_j(x)=e^{{\rm i}\zeta_j\cdot x}(1+r_j(x)),
 \end{align}
 where $r_j\in L^2(\widetilde\Omega)$ satisfies
 \begin{align}\label{m7}
 	\|r_j\|_{L^2(\widetilde\Omega)}\leq\frac{C_1}{|{\rm Im}\zeta_j|}\|q\|_{L^\infty(\Omega)}, \quad j=1,2.
 \end{align}
Now, we are ready to prove Theorem \ref{mm}.

\begin{proof}[Proof of Theorem \ref{mm}]
We will combine the ideas of \cite{garcia2021runge} and \cite{krupchyk2019stability}, \cite{joud2009stability}.
We first recall the three open subsets $\mathcal O_j (j=1,2,3)$ of $\mathcal O$, which are also neighborhood of $\partial\Omega$ with $C^{\infty}$ boundary,  such that
$$\partial\Omega\subset \partial\mathcal O_j,\quad  \overline{\mathcal O}_{j+1}\subset \mathcal O_j, \quad \overline{\mathcal O}_j\subset \mathcal O, \quad  j=1,2,3.$$
Let $k\geq1$ and let $u_2\in H^2(\widetilde\Omega)$ satisfy (\ref{m1}) with the form (\ref{m2}).

Next, we seek to approximate $u_2$ up to some order $\epsilon>0$ which will be chosen later. We apply Lemma \ref{le4} with $\Omega_2=\Omega$, $\Omega_1=\Omega\backslash\overline{\mathcal O}_1$ and $\widetilde\Omega_1=\Omega\backslash\overline{\mathcal O}_2$. This yields solutions $\widetilde u_2\in H^1(\Omega)$ to
\begin{align*}
(-\Delta-k^2+q_2)\widetilde u_2=0 \quad \mbox{in} ~~\Omega,	
\end{align*}
with $\widetilde u_2|_{\partial \Omega}\in \widetilde H^{\frac{1}{2}}(\Gamma_D)$ supported in $\Gamma_D$ and
\begin{align}\label{m4}
	\|\widetilde u_2-u_2\|_{L^2(\Omega\backslash{\overline{\mathcal O}_1})}\leq\epsilon\|u_2\|_{H^1(\Omega\backslash{\overline{\mathcal O}_2})},
\end{align}
\begin{align}\label{m5}
	\|\widetilde u_2\|_{H^{\frac{1}{2}}(\Gamma_D)}\leq Ce^{Ck}\epsilon^{-\mu}\|u_2\|_{L^2(\Omega\backslash{\overline{\mathcal O}_1})},
\end{align}
with $C>1$.
Further, we set $f=\widetilde u_2|_{\partial \Omega}\in \widetilde H^{\frac{1}{2}}(\Gamma_D)$, then supp$f\subset\Gamma_D$.

Let $v\in H^1(\Omega)$ be the solution to the following problem
\begin{align*}
	(-\Delta-k^2+q_1)v&=0 \quad\mbox{in} ~~\Omega,\\
	v&=f\quad\mbox{on}~~\partial\Omega.
\end{align*}
Then, by defining $u=v-\widetilde u_2\in H^1(\Omega)$, we get
\begin{align}\label{a3}
	\begin{split}
		(-\Delta-k^2+q_1)u&=(q_2-q_1)\widetilde u_2 \quad\mbox{in} ~~\Omega,\\
	u&=0\quad\mbox{on}~~\partial\Omega.
	\end{split}
\end{align}
Since $q_1, q_2\in L^{\infty}(\Omega)$ and $\widetilde u_2\in H^1(\Omega), u\in H^1(\Omega)$, we can see that $u\in H^2(\Omega)$ by the theory of elliptic equations.
To proceed, we introduce a cut-off function $\chi_2\in C_0^\infty(\Omega)$ satisfying $0\leq\chi_2\leq1$ and
\begin{align}\label{m6}
	\chi_2(x)=
	\begin{cases}
		0,\quad x\in\mathcal O_3,\\
		1,\quad x\in \overline\Omega\backslash\mathcal O_2.
	\end{cases}
\end{align}
Thus, if we put $\widetilde u=\chi_2 u$, we see that $\widetilde u$ is the solution of the following problem
\begin{align}\label{m8}
	\begin{split}
		(-\Delta-k^2+q_1)\widetilde u&=(q_1-q_2)\widetilde u_2+[-\Delta,\chi_2]u \quad\mbox{in} ~~\Omega,\\
	\widetilde u&=0\quad\mbox{on}~~\partial\Omega,
	\end{split}
\end{align}
where we have used $\chi_2(q_1-q_2)=q_1-q_2$ in $\Omega$ by noting $q_1=q_2$ in ${\mathcal O}$ and (\ref{m6}).

Next, we recall that $u_1$ is a solution to
$$(-\Delta-k^2+q_1)u_1=0 \quad \mbox{in} ~~\widetilde\Omega,$$
 in the form $$
 u_1(x)=e^{{\rm i}\zeta_1\cdot x}(1+r_1(x)),$$
and $r_1(x)$ satisfies (\ref{m7}) with $j=1$. By applying Lemma \ref{le2} to (\ref{m1}), we can see that $u_1\in H^2(\Omega)$.

Multiplying the first equation of (\ref{m8}) by $u_1$ and integrating by parts over $\Omega$, we obtain
\begin{align}\label{a9}
	\int_\Omega(q_2-q_1)u_1\widetilde u_2\,dx=\int_\Omega[-\Delta,\chi_2]u u_1\,dx,
\end{align}
where we have noted that $\widetilde u_2|_{\partial\Omega}=\partial_{\nu}\widetilde u_2|_{\partial\Omega}=0$ by the definition of $\chi_2$ in (\ref{m6}).
Further, we apply Proposition \ref{pp} and use (\ref{m6}) to get
\begin{align}\label{a2}
	\left|\int_\Omega[-\Delta,\chi_2]u u_1\,dx\right|&\leq \|[-\Delta,\chi_2]u\|_{L^2(\mathcal O_2\backslash\mathcal O_3)}\|u_1\|_{L^2(\Omega)}\nonumber\\
	&\leq C\|u\|_{H^1(\mathcal O_2\backslash\mathcal O_3)}\|u_1\|_{L^2(\Omega)}\nonumber\\
	&\leq C\left(e^{-\frac{\alpha_1}{h}}\|u\|_{H^1(\Omega)}+e^{\frac{\alpha_2}{h}}\|u\|_{H^2(\Omega)}^{\frac{1}{2}}\|\partial_\nu u\|_{H^{-\frac{1}{2}}(\Gamma_N)}^{\frac{1}{2}}\right)\|u_1\|_{L^2(\Omega)},
\end{align}
where $\alpha_1>0$, $\alpha_2>0$ and $0<h\leq\frac{h_0}{k}$ for $k\geq1$, $0<h_0\leq \min \left\{1, \frac{\alpha_1}{2}\right\}$.

By using the mapping properties of the partial DtN map (\ref{a1}), the trace theorem and (\ref{m5}), we obtain
\begin{align}\label{a4}
	\|\partial_\nu u\|_{H^{-\frac{1}{2}}(\Gamma_N)}&=\|\partial_\nu (v-\widetilde u_2)\|_{H^{-\frac{1}{2}}(\Gamma_N)}\nonumber\\
	&=\|\left(\Lambda_{q_1}^{\Gamma_D,\Gamma_N}(k)-\Lambda_{q_2}^{\Gamma_D,\Gamma_N}(k)\right)\widetilde u_2\|_{H^{-\frac{1}{2}}(\Gamma_N)}\nonumber\\
	&\leq C\|\Lambda_{q_1}^{\Gamma_D,\Gamma_N}(k)-\Lambda_{q_2}^{\Gamma_D,\Gamma_N}(k)\|_{\widetilde H^{\frac{1}{2}}(\Gamma_D) \rightarrow H^{-\frac{1}{2}}(\Gamma_N)}\|\widetilde u_2\|_{H^{\frac{1}{2}}(\Gamma_D)}\nonumber\\
	&\leq Ce^{Ck}\epsilon^{-\mu}\|\Lambda_{q_1}^{\Gamma_D,\Gamma_N}(k)-\Lambda_{q_2}^{\Gamma_D,\Gamma_N}(k)\|_{\widetilde H^{\frac{1}{2}}(\Gamma_D) \rightarrow H^{-\frac{1}{2}}(\Gamma_N)}\|u_2\|_{L^2(\Omega)},
\end{align}
 with $C>1$.

 Let us now bound $\|u\|_{H^1(\Omega)}$ in (\ref{a2}). First, we recall that $u$ satisfies (\ref{a3}) and use Lemma \ref{le3}. We get that for all $k\geq1$ satisfying (a2),
$$
	\|u\|_{H^1(\Omega)}\leq Ck^{n+1}\|(q_2-q_1)\widetilde u_2\|_{L^2(\Omega)}\leq Ck^{n+1}\|\widetilde u_2\|_{L^2(\Omega\backslash\overline{\mathcal O}_1)}.
$$
In the last inequality, we have used  $q_1=q_2$ in ${\mathcal O}$.
Then, in view of (\ref{m4}), we have
\begin{align*}
	\|\widetilde u_2\|_{L^2(\Omega\backslash\overline{\mathcal O}_1)}&\leq\|\widetilde u_2-u_2\|_{L^2(\Omega\backslash{\overline{\mathcal O}_1})}+\|u_2\|_{L^2(\Omega\backslash{\overline{\mathcal O}_1})}\\
&\leq\epsilon\|u_2\|_{H^1(\Omega)}+\|u_2\|_{L^2(\Omega)}.
\end{align*}
Thus, we obtain
\begin{align}\label{a5}
	\|u\|_{H^1(\Omega)}\leq Ck^{n+1}\left(\epsilon\|u_2\|_{H^1(\Omega)}+\|u_2\|_{L^2(\Omega)}\right).
\end{align}

As for the estimate of $\|u\|_{H^2(\Omega)}$, by a priori estimate for (\ref{a3})  and Lemma \ref{le3}, we get
\begin{align}\label{a6}
\|u\|_{H^2(\Omega)}&\leq C\|(q_2-q_1)\widetilde u_2+(k^2-q_1)u\|_{L^2(\Omega)}\nonumber\\
&\leq C\left(\|(q_2-q_1)\widetilde u_2\|_{L^2(\Omega)}+k^{2}\| u\|_{L^2(\Omega)}\right)\nonumber\\
&\leq C\left(\|(q_2-q_1)\widetilde u_2\|_{L^2(\Omega)}+k^{n+2}\|(q_2-q_1)\widetilde u_2\|_{L^2(\Omega)}\right)\nonumber\\
&\leq Ck^{n+2}\|\widetilde u_2\|_{L^2(\Omega\backslash\overline{\mathcal O}_1)}\nonumber\\
&\leq Ck^{n+2}\left(\epsilon\|u_2\|_{H^1(\Omega)}+\|u_2\|_{L^2(\Omega)}\right).
\end{align}
In the fourth inequality, we have also used  $q_1=q_2$ in ${\mathcal O}$.

Let $R>0$ be such that $\widetilde\Omega$ is contained in a ball centered at zero of radius $R$. Then thanks to (\ref{m1}), (\ref{m2}) and (\ref{m7}), we have
\begin{align}\label{a7}
	\|u_j\|_{L^2(\Omega)}\leq \|u_j\|_{L^2(\widetilde\Omega)}\leq Ce^{aR}, \quad j=1,2.
\end{align}
Furthermore, by applying Lemma \ref{le2} to (\ref{m1}),
 we get
 \begin{align}\label{a8}
 	\|u_j\|_{H^1(\Omega)}\leq Ck\|u_j\|_{L^2(\widetilde\Omega)}\leq Cke^{aR} ,\quad j=1,2.
 \end{align}
Collecting the identities and the estimates  (\ref{a2})-(\ref{a8}), we obtain that
\begin{align}\label{a12}
	\left|\int_\Omega[-\Delta,\chi_2]u u_1\,dx\right|\leq  Ce^{2aR-\frac{\alpha_1}{h}}k^{n+1}(\epsilon k+1)+Ce^{2aR+\frac{\alpha_2}{h}} e^{Ck}\epsilon^{-\frac{\mu}{2}}\delta^{\frac{1}{2}}(\epsilon^{\frac{1}{2}}k^{\frac{1}{2}}+1),
\end{align}
for some constant $\mu>1$ and all $k\geq1$ satisfying (a2), $0<h\leq\frac{h_0}{k}$, $h_0\in \left(0, \min \left\{1, \frac{\alpha_1}{2}\right\}\right]$, $a\geq\max\{C_0M,1\}$, $\epsilon>0$. Here, we abbreviate $\delta:=\|\Lambda_{q_1}^{\Gamma_D,\Gamma_N}(k)-\Lambda_{q_2}^{\Gamma_D,\Gamma_N}(k)\|_{\widetilde H^{\frac{1}{2}}(\Gamma_D) \rightarrow H^{-\frac{1}{2}}(\Gamma_N)}$.

Now it follows from (\ref{a9}) and (\ref{a12}) by choosing $0<\epsilon<1$ that
\begin{align}\label{a10}
	\left|\int_\Omega(q_2-q_1)u_1\widetilde u_2\,dx\right|\leq C\left(e^{2aR-\frac{\alpha_1}{h}}k^{n+2}+e^{2aR+\frac{\alpha_2}{h}} e^{Ck}\epsilon^{-\frac{\mu}{2}}\delta^{\frac{1}{2}}\right),
\end{align}
for  all $k\geq1$ satisfying (a2), $0<h\leq\frac{h_0}{k}$, $h_0\in \left(0, \min \left\{1, \frac{\alpha_1}{2}\right\}\right]$, $a\geq\max\{C_0M,1\}$, $0<\epsilon<1$.
Take $h=\frac{1}{a\widetilde{\gamma} }$ in (\ref{a10}) and choose a constant $\widetilde{\gamma} >0$ sufficiently large so that
$$e^{2aR-\frac{\alpha_1}{2}a\widetilde{\gamma}}\leq e^{-\alpha_3 a}$$
and
\begin{align}\label{q}
	e^{2aR+\alpha_2a\widetilde{\gamma}}\leq e^{\alpha_4 a},
\end{align}
for some constants $\alpha_3>0$ and $\alpha_4>0$. It is easy to see that $0<h\leq\frac{h_0}{k}$ implies that $a\geq \frac{1}{\widetilde{\gamma} h_0}k$.

We conclude from (\ref{a10}) that for all $k\geq1$ satisfying (a2), $a\geq\max\left\{C_0M, \frac{1}{\widetilde{\gamma} h_0}k, 1\right\}$, $0<\epsilon<1$,
\begin{align*}
	\left|\int_\Omega(q_2-q_1)u_1\widetilde u_2\,dx\right|\leq C\left(e^{-\frac{\alpha_1}{2}a\widetilde{\gamma}-\alpha_3a}k^{n+2}+e^{\alpha_4 a} e^{Ck}\epsilon^{-\frac{\mu}{2}}\delta^{\frac{1}{2}}\right).
\end{align*}
By $h_0\leq \frac{\alpha_1}{2}$ and $a\geq \frac{1}{\widetilde{\gamma} h_0}k$, we have $e^{-\frac{\alpha_1}{2}a\widetilde{\gamma}}\leq e^{-h_0 a\widetilde{\gamma} }\leq e^{-k}$. Consequently, we obtain
\begin{align}\label{Tjbds1}
	\left|\int_\Omega(q_2-q_1)u_1\widetilde u_2\,dx\right|&\leq C\left(e^{-\frac{\alpha_1}{2}a\widetilde{\gamma}-\alpha_3a}k^{n+2}+e^{\alpha_4 a} e^{Ck}\epsilon^{-\frac{\mu}{2}}\delta^{\frac{1}{2}}\right)\notag\\
     &\leq C\left(e^{-\alpha_3a}e^{-k}k^{n+2}+e^{\alpha_4 a} e^{Ck}\epsilon^{-\frac{\mu}{2}}\delta^{\frac{1}{2}}\right)\notag\\
     &\leq C\left(e^{-\alpha_3a}+e^{\alpha_4 a} e^{Ck}\epsilon^{-\frac{\mu}{2}}\delta^{\frac{1}{2}}\right)\notag\\
     &\leq C\left(\frac{1}{a}+e^{\alpha_4 a} e^{Ck}\epsilon^{-\frac{\mu}{2}}\delta^{\frac{1}{2}}\right),
\end{align}
for all $k\geq1$ satisfying (a2), $a\geq\max\left\{C_0M, \frac{1}{\widetilde{\gamma} h_0}k, 1\right\}$, $0<\epsilon<1$.

Recalling that we extended $q_1$, $q_2$ by zero to $\mathbb R^n$, thus, we seek to apply the previous results to estimate $$\left|\mathcal F(q_2-q_1)(\xi)\right|=\left|\int_\Omega(q_2-q_1)e^{-{\rm i}\xi\cdot x}\,dx\right|,$$
for any $|\xi|\leq2\sqrt{k^2+a^2}$. Notice that
\begin{align*}
	e^{-{\rm i}\xi\cdot x}&=u_1u_2-e^{-{\rm i}\xi\cdot x}
(r_1+r_2+r_1r_2)\\
&=u_1\widetilde u_2+u_1(u_2-\widetilde u_2)-e^{-{\rm i}\xi\cdot x}
(r_1+r_2+r_1r_2).
\end{align*}
Thus, we have
\begin{align}\label{z1}
	\left|\mathcal F(q_2-q_1)(\xi)\right|\leq &\left|\int_\Omega(q_2-q_1)u_1\widetilde u_2\,dx\right|+\left|\int_\Omega(q_2-q_1)u_1(u_2-\widetilde u_2)\,dx\right|\nonumber\\
	&+\left|\int_\Omega(q_2-q_1)e^{-{\rm i}\xi\cdot x}
(r_1+r_2+r_1r_2)\,dx\right|\nonumber\\
:=&R_1+R_2+R_3.
\end{align}

By (\ref{Tjbds1}), $R_1$ satisfies
\begin{align}\label{z2}
	R_1\leq C\left(\frac{1}{a}+e^{\alpha_4 a} e^{Ck}\epsilon^{-\frac{\mu}{2}}\delta^{\frac{1}{2}}\right).
\end{align}

As for $R_2$, we use the Runge approximation bound (\ref{m4}) again and invoke the estimates (\ref{a7}), (\ref{a8}) to get that
\begin{align}
	R_2\leq C\|u_1\|_{L^2(\Omega\backslash\overline{\mathcal O}_1)}\|u_2-\widetilde u_2\|_{L^2(\Omega\backslash\overline{\mathcal O}_1)}	\leq Ce^{2aR}\epsilon k\leq Ce^{\alpha_4 a}\epsilon k,
\end{align}
by noting $q_1=q_2$ in $\Omega\backslash\overline{\mathcal O}_1$ and (\ref{q}).

Further, applying the previous estimate (\ref{m7}), we obtain
\begin{align}\label{z3}
	R_3\leq C\left(\|r_1\|_{L^2(\widetilde\Omega)}+\|r_2\|_{L^2(\widetilde\Omega)}+\|r_1\|_{L^2(\widetilde\Omega)}\|r_2\|_{L^2(\widetilde\Omega)}\right)\leq\frac{C}{a}.
\end{align}
It follows from (\ref{z1}) with the help of (\ref{z2})-(\ref{z3}) that

\begin{align}\label{z3O}
	\left|\mathcal F(q_2-q_1)(\xi)\right|\leq C\left(e^{\alpha_4 a}\epsilon k+e^{\alpha_4 a} e^{Ck}\epsilon^{-\frac{\mu}{2}}\delta^{\frac{1}{2}}+\frac{1}{a}\right),
\end{align}
for all $k\geq1$ satisfying (a2), $a\geq\max\left\{C_0M, \frac{1}{\widetilde{\gamma} h_0}k, 1\right\}$, $\xi\in\mathbb R^n$ such that $|\xi|\leq2\sqrt{k^2+a^2}$ and $0<\epsilon<1$ to be chosen.

In order to estimate $\|q_2-q_1\|_{H^{-1}(\Omega)}$, we
take $\tau\leq2\sqrt{k^2+a^2}$ to be chosen and use (\ref{z3O}), together with Parseval's formula to  get
\begin{align}\label{z4}
	\|q_2-q_1\|_{H^{-1}(\Omega)}^2&\leq\int_{|\xi|\leq\tau}\frac{\left|\mathcal F(q_2-q_1)(\xi)\right|^2}{1+|\xi|^2}\,d\xi+\int_{|\xi|\geq\tau}\frac{\left|\mathcal F(q_2-q_1)(\xi)\right|^2}{1+|\xi|^2}\,d\xi\nonumber\\
	&\leq C\tau^n\left(e^{2\alpha_4 a}\epsilon^2 k^2+e^{2\alpha_4 a} e^{Ck}\epsilon^{-\mu}\delta+\frac{1}{a^2}\right)+\frac{C}{\tau^2}.
\end{align}
Setting $\tau=a^{\frac{2}{n+2}}$, (\ref{z4}) gives that
\begin{align*}
	\|q_2-q_1\|_{H^{-1}(\Omega)}^2&\leq C\left(a^{\frac{2n}{n+2}}e^{2\alpha_4 a}\epsilon^2 k^2+a^{\frac{2n}{n+2}}e^{2\alpha_4 a} e^{Ck}\epsilon^{-\mu}\delta+a^{-\frac{4}{n+2}}\right)\\
	&\leq C\left(e^{6\alpha_4 a}\epsilon^2+e^{4\alpha_4 a} e^{Ca}\epsilon^{-\mu}\delta+a^{-\frac{4}{n+2}}\right),
\end{align*}
where we used $k\leq a\widetilde{\gamma} h_0$. Noting $0<\delta<1$, we choose
 $$\epsilon=\delta^{\frac{1}{\mu+2}}\in (0,1),$$
which results in
\begin{align*}
	\|q_2-q_1\|_{H^{-1}(\Omega)}^2\leq C\left(e^{(6\alpha_4+C)a}\delta^{\frac{2}{\mu+2}}+a^{-\frac{4}{n+2}}\right).
\end{align*}
Recalling that $a\geq\max\left\{C_0M, \frac{1}{\widetilde{\gamma} h_0}k, 1\right\}$ and setting $$a=\frac{1}{6\alpha_4+C}\log{(\delta^{-\frac{1}{\mu+2}})}+\frac{1}{\widetilde{\gamma} h_0}k+C_0M+1,$$
we obtain
\begin{align}\label{z5}
	\|q_2-q_1\|_{H^{-1}(\Omega)}^2\leq \frac{C}{\left(k+\log\frac{1}{\delta}\right)^{\frac{4}{n+2}}}+e^{Ck}\delta^\frac{1}{\mu+2},
\end{align}
where $\mu>1$.
Applying Young's inequality, we can estimate the last term as follows:
$$e^{Ck}\delta^\frac{1}{\mu+2}\leq C\left(e^{Ck}\delta^2+k^{-1}\delta^\frac{2}{4\mu+7}\right).$$
Taking into account that
$$k^{-1}\delta^\beta\leq \left(k+\beta\log\frac{1}{\delta}\right)^{-\frac{4}{n+2}}$$
for $0<\beta<1$, we conclude from (\ref{z5}) that for all $k\geq 1$ satisfying (a2),
\begin{align*}
	\|q_2-q_1\|_{H^{-1}(\Omega)}\leq e^{Ck}\delta+\frac{C}{\left(k+\log\frac{1}{\delta}\right)^{\frac{2}{n+2}}}.
\end{align*}
This completes the proof of Theorem \ref{th}.

\end{proof}

\begin{proof}[Proof of Corollary \ref{co}]We follow the classical argument due to \cite{alessandrini1988stable}, see also \cite{caro2016stability}. Let $\varepsilon>0$ be such that $s=\frac{n}{2}+2\varepsilon$. Then by the Sobolev embedding, interpolation and the a priori bounds for $q_j$, we get
\begin{align*}
	\|q_2-q_1\|_{L^\infty(\Omega)}\leq & C\|q_2-q_1\|_{H^{\frac{n}{2}+\varepsilon}(\Omega)}\\
	\leq &C\|q_2-q_1\|_{H^{-1}(\Omega)}^{\frac{\varepsilon}{1+s}}\|q_2-q_1\|_{H^s(\Omega)}^{\frac{1-\varepsilon+s}{s+1}}\\
	\leq &C(2M)^{\frac{1-\varepsilon+s}{s+1}}\|q_2-q_1\|_{H^{-1}(\Omega)}^{\frac{\varepsilon}{1+s}}\\
	\leq &\left(e^{Ck}\delta+\frac{C}{\left(k+\log\frac{1}{\delta}\right)^{\frac{2}{n+2}}}\right)^{\frac{s-\frac{n}{2}}{2(s+1)}}.
\end{align*}
This completes the proof of Corollary \ref{co}.	
\end{proof}

\appendix

\section{proof of Lemma \ref{c1}}
We adapt slightly the proof of \cite[Theorem 2.1]{krupchyk2019stability} to our case.
We shall proceed by following the arguments of \cite{Fursikov1996Control}, as presented in \cite[Theorem 4.3.9]{le2012carleman}, \cite[Theorem 2.1]{krupchyk2019stability}. By density, it is suffices to prove (\ref{app}) for $v\in C^\infty(\overline\Omega)$. We write
$$P_\Phi(h,E)=A_2+{\rm i}A_1,$$
where $\Phi':=\nabla\Phi$ and
$$A_2=(hD)^2-|\Phi'|^2-E,\quad A_1=2\Phi'\cdot hD-{\rm i}h\Delta\Phi.$$
Here $D=\frac{1}{{\rm i}}\nabla $. The idea of Fursikov and Imanuvilov \cite{Fursikov1996Control} is the following: rather than considering the equation $P_\Phi(h,E)v=g$, one works with $$(A_2+{\rm i}\underline{A_1})v=g+h\theta\Delta\Phi u,$$where $\theta>0$ is to be chosen and $$\underline{A_1}:=\frac{1}{{\rm i}}[2\Phi'\cdot h\nabla+h(\theta+1)\Delta\Phi].$$
Following \cite{Fursikov1996Control}, \cite[Theorem 4.3.9]{le2012carleman}, we get

\begin{align}\label{1}
\begin{split}
	\|g+h\theta\Delta\Phi v\|^2_{L^2(\Omega)}&=\|A_2v\|^2_{L^2(\Omega)}+\|\underline{A_1}v\|^2_{L^2(\Omega)}+2{\rm Re}(A_2v,{\rm i}\underline{A_1}v)_{L^2(\Omega)}\\
	&\geq 2{\rm Re}(A_2v,{\rm i}\underline{A_1}v)_{L^2(\Omega)}.
\end{split}
\end{align}

We shall next compute
\begin{align}\label{2}
{\rm Re}(A_2v,{\rm i}\underline{A_1}v)_{L^2(\Omega)}={\rm Re}\int_\Omega((hD)^2v-|\Phi'|^2v-Ev)(2\Phi'\cdot h\nabla \overline v+h(\theta+1)\Delta\Phi \overline v)\,dx	
\end{align}
In doing so, as in \cite[Theorem 4.3.9]{le2012carleman}, we write the integral in (\ref{2}) as a sum of six terms $I_{jk}$, $1\leq j\leq 3$, $1\leq k\leq 2$, where
$I_{jk}$ is the $L^2$ scalar product of the $j$th term in the expression of $A_2v$ and the $k$th term in the expression of
${\rm i}\underline{A_1}v$. Furthermore, we recall the extra assumption $v|_{\partial\Omega}=0$ in our Lemma.

For the term $I_{11}$ in (\ref{2}), performing two integration by parts, as in  \cite[Theorem 4.3.9]{le2012carleman}, we get
\begin{align*}
	I_{11}&={\rm Re}\int_\Omega(-h^2\Delta v) (2\Phi'\cdot h\nabla \overline v)\,dx=2h^3\int_\Omega\Phi''\nabla v\cdot\nabla \overline v\,dx\\
	&-h^3\int_\Omega\Delta\Phi|\nabla v|^2\,dx-h^3\int_{\partial\Omega}\partial_\nu\Phi|\nabla v|^2\,dS+2h^3Re\int_{\partial\Omega}(\partial_\nu v)\Phi'\cdot\nabla\overline v\,dS.
\end{align*}
For the term $I_{12}$ in (\ref{2}), performing an integration by parts, as in \cite[Theorem 4.3.9]{le2012carleman}, we obtain that

\begin{align*}
	I_{12}=&{\rm Re}\int_\Omega(-h^2\Delta v) h(\theta+1)(\Delta\Phi) \overline v\,dx\\=&h^3(\theta+1)\int_\Omega\Delta\Phi|\nabla v|^2\,dx+h^3(\theta+1)Re\int_\Omega(\nabla v\cdot\nabla\Delta\Phi)\overline u\,dx.
\end{align*}
For the term $I_{21}$ in (\ref{2}), proceeding as in \cite[Theorem 4.3.9]{le2012carleman}, and performing an integration by parts, we get
\begin{align*}
	I_{21}=-2{\rm Re}\int_\Omega|\Phi'|^2v\Phi'\cdot h\nabla\overline v\,dx=h\int_\Omega\nabla\cdot(|\Phi'|^2\Phi')|v|^2\,dx.
\end{align*}
For the term $I_{22}$ in (\ref{2}), we have
\begin{align*}
	I_{22}=-{\rm Re}\int_\Omega|\Phi'|^2 vh(\theta+1)(\Delta\Phi)\overline v\,dx=-h(\theta+1)\int_\Omega(\Delta\Phi)|\Phi'|^2|v|^2\,dx.
\end{align*}
Finally, using that $\Phi'\cdot\nabla|v|^2=2{\rm Re}(v\Phi'\cdot\nabla\overline v)$, and integrating by parts, we get
\begin{align*}
	I_{31}+I_{32}&=-{\rm Re}\int_\Omega Ev(2\Phi'\cdot h\nabla\overline v+h(\theta+1)\Delta\Phi\overline v)\,dx\\
	&=-hE\int_\Omega\Phi'\cdot\nabla|v|^2\,dx-hE(\theta+1)\int_\Omega\Delta\Phi|v|^2\,dx\\
	&=-hE\theta\int_\Omega\Delta\Phi|v|^2\,dx.
\end{align*}
Collecting all the terms together, we obtain that
\begin{align}\label{3}
	{\rm Re}(A_2v,{\rm i}\underline{A_1}v)_{L^2(\Omega)}=h\int_\Omega\widetilde\alpha_0|v|^2\,dx+h^3\int_\Omega\alpha_1|\nabla v|^2\,dx+X+bt,
\end{align}
where
\begin{align}\label{4}
&\widetilde\alpha_0=\alpha_0-E\theta\Delta\Phi,\quad \alpha_0=\nabla\cdot(|\Phi'|^2\Phi')-(\theta+1)(\Delta\Phi)|\Phi'|^2, \quad \alpha_1=\theta\Delta\Phi,\nonumber\\	
&X=2h^3\int_\Omega\Phi''\nabla v\cdot\nabla\overline v\,dx+h^3(\theta+1){\rm Re}\int_\Omega(\nabla\Delta\Phi\cdot\nabla v)\overline v\,dx,\nonumber\\
&bt=-h^3\int_{\partial\Omega}\partial_\nu\Phi|\nabla v|^2\,dS+2h^3Re\int_{\partial\Omega}(\partial_\nu v)\Phi'\cdot\nabla\overline v\,dS.
\end{align}
Now by Lemma 4.3.10 in \cite{le2012carleman}, we have
\begin{align}\label{5}
	\alpha_0\geq C\gamma^4\Phi^3,
\end{align}
provided $\theta<2$. Assuming that $\gamma\geq1$ and using that $\Psi\geq0$ on $\overline\Omega$, we get for all $0\leq E\leq1$,
\begin{align}\label{6}
	|E\theta\Delta\Phi|\leq|\theta(\gamma^2|\Psi'|^2\Phi+\gamma\Delta\Psi\Phi)|\leq C\theta\gamma^3\Phi^3.
\end{align}
It follows from (\ref{5}) and (\ref{6}) that $\widetilde\alpha_0\geq C\gamma^4\Phi^3$ for $\gamma>1$ sufficiently large. As $\theta>0$, we also have $\alpha_1\geq C\gamma^2\Phi$ for $\gamma$ sufficiently large. Hence, fixing $\theta=1$, we conclude from (\ref{3}), by absorbing the remainder term $X$ as explained in \cite{le2012carleman}, that for all $h>0$ small enough, all $\gamma$ large enough, and $0\leq E\leq 1$,
\begin{align}\label{7}
	{\rm Re}(A_2v,{\rm i}\underline{A_1}v)_{L^2(\Omega)}\geq Ch\gamma^4\int_\Omega\Phi^3|v|^2\,dx+Ch\gamma^2\int_\Omega\Phi|h\nabla v|^2\,dx-|bt|.
\end{align}
It follows from (\ref{4}) and $v|_{\partial\Omega}=0$ which implies $|\nabla v|=|\partial_\nu v|$ that for all $0\leq E\leq 1$,
\begin{align}\label{8}
	bt=&-h^3\int_{\partial\Omega}\gamma\Phi\partial_\nu\Psi|\partial_\nu v|^2\,dS+2h^3Re\int_{\partial\Omega}(\partial_\nu v)\gamma\Phi\nabla\Psi\cdot(\nu\partial_\nu\overline v)\,dS\nonumber\\
	=&-h^3\int_{\partial\Omega}\gamma\Phi\partial_\nu\Psi|\partial_\nu v|^2\,dS+2h^3Re\int_{\partial\Omega}\gamma\Phi(\nabla\Psi\cdot\nu)|\partial_\nu v|^2\,dS\nonumber\\
	=&h\gamma\int_{\partial\Omega}\Phi\partial_\nu\Psi|h\partial_\nu v|^2\,dS.
\end{align}
Combining (\ref{1}), (\ref{7}) and (\ref{8}) and absorbing the term $h^2\theta^2\|\Delta\Phi v\|^2_{L^2(\Omega)}$ by choosing $h$ small enough independent of $\gamma$, we get (\ref{app}). This completes the proof of Lemma \ref{c1}.

\end{document}